\documentclass[11pt,letterpaper]{amsart}

\newtheorem{theorem}{Theorem}[section]
\newtheorem{proposition}[theorem]{Proposition}

\newtheorem{lemma}[theorem]{Lemma}

\theoremstyle{definition}
\newtheorem{definition}[theorem]{Definition}

\newtheorem{abctheorem}{Theorem}

\newtheorem{abcproblem}[abctheorem]{Problem}

\DeclareMathOperator{\E}{E}

\newcommand{\Q}{\mathbb{Q}}
\newcommand{\Z}{\mathbb{Z}}
\newcommand{\C}{\mathbb{C}}
\newcommand{\e}{\epsilon}

\newcommand{\abs}[1]{\lvert#1\rvert}
\newcommand{\bigabs}[1]{\big\lvert#1\big\rvert}

\newcommand{\biggabs}[1]{\bigg\lvert#1\bigg\rvert}
\newcommand{\Biggabs}[1]{\Bigg\lvert#1\Bigg\rvert}

\newcommand{\sums}[1]{\sum_{\substack{#1}}}

\begin{document}

\title[The correlation measures of finite sequences]{$\!$The correlation measures of finite sequences:$\!$ limiting distributions and minimum values}

\author{Kai-Uwe Schmidt}

\date{10 January 2014 (revised 06 January 2015)}

\subjclass[2010]{Primary: 11K45; Secondary 60C05, 68R15}

\address{Faculty of Mathematics, Otto-von-Guericke University, Universit\"atsplatz~2, 39106 Magdeburg, Germany.}

\email{kaiuwe.schmidt@ovgu.de}

\begin{abstract}
Three measures of pseudorandomness of finite binary sequences were introduced by Mauduit and S\'ark\"ozy in 1997 and have~been studied extensively since then: the normality measure, the well-distri\-bution measure, and the correlation measure of order $r$. Our main result is that the correlation measure of order $r$ for random binary sequences converges strongly, and so has a limiting distribution. This solves a problem due to Alon, Kohayakawa, Mauduit, Moreira, and R\"odl. We also show that the best known lower bounds for the minimum values of the correlation measures are simple consequences of a celebrated result due to Welch, concerning the maximum nontrivial scalar products over a set of vectors.
\end{abstract}

\maketitle


\section{Introduction and main results}

We consider finite binary sequences, namely elements $A_n$ of $\{-1,1\}^n$. Mauduit and S\'ark\"ozy~\cite{Mauduit1997} introduced three measures of pseudorandomness for finite binary sequences: the \emph{well distribution measure} $W(A_n)$, the \emph{normality measure} $\mathcal{N}(A_n)$, and the \emph{$r$-th order correlation measure} $C_r(A_n)$. These measures have been studied extensively (see~\cite{Mauduit1997},~\cite{Cassaigne2002},~\cite{Alon2006},~\cite{Alon2007},~\cite{Ahlswede2008},~\cite{Aistleitner2013},~\cite{Aistleitner2014}, for example). Finite binary sequences for which these measures are small are considered to possess a high `level of randomness'.
\par
In this paper, we are concerned with the correlation measures of finite binary sequences. Let $A_n=(a_1,a_2,\dots,a_n)$ be an element of $\{-1,1\}^n$. For $2\le r \le n$, the \emph{$r$-th order correlation measure} of $A_n$ is defined as
\[
C_r(A_n)=\max_{0\le u_1<u_2<\cdots<u_r<n}\;\max_{1\le m\le n-u_r}\Biggabs{\sum_{j=1}^ma_{j+u_1}a_{j+u_2}\cdots a_{j+u_r}}.
\]
\par
Following earlier work by Cassaigne, Mauduit, and S\'ark\"ozy~\cite{Cassaigne2002}, Alon, Kohayakawa, Mauduit, Moreira, and R\"odl~\cite{Alon2007} studied the behaviour of $W(A_n)$, $\mathcal{N}(A_n)$, and $C_r(A_n)$ when $A_n$ is drawn at random from $\{-1,1\}^n$, equipped with the uniform probability measure. They posed the following problem.
\begin{abcproblem}[{\cite[Problem 33]{Alon2007}}]
\label{probA}
Investigate the existence of the limiting distributions of
\begin{gather}
\left\{\frac{W(A_n)}{\sqrt{n}}\right\}_{n\ge 1}\quad\text{and}\quad
\left\{\frac{\mathcal{N}(A_n)}{\sqrt{n}}\right\}_{n\ge 1}   \nonumber\\
\intertext{and}
\left\{\frac{C_r(A_n)}{\sqrt{n\log \binom{n}{r}}}\right\}_{n\ge r}.\label{eqn:_question_limit_distribution}
\end{gather}
Investigate these distributions.
\end{abcproblem}
\par
The first two instances of Problem~\ref{probA} have been solved recently: Aistleitner~\cite{Aistleitner2013},~\cite{Aistleitner2014} proved that the limiting distributions of $W(A_n)/\sqrt{n}$ and of $\mathcal{N}(A_n)/\sqrt{n}$ exist. Moreover, a tail characterisation of the limiting distribution of $W(A_n)/\sqrt{n}$ is provided in~\cite{Aistleitner2013}. It is known that, if~\eqref{eqn:_question_limit_distribution} has a limiting distribution, then it is a Dirac measure~\cite[Theorem 3]{Alon2007}. We shall resolve the third instance  of Problem~\ref{probA} by proving strong convergence of~\eqref{eqn:_question_limit_distribution}. To do so, we consider the set $\Omega$ of infinite sequences of elements $-1$ or $1$ and endow $\Omega$ in the standard way with the probability measure defined by
\begin{equation}
\Pr\big[(a_1,a_2,\dots)\in\Omega:a_1=c_1,a_2=c_2,\dots,a_n=c_n\big]=2^{-n}   \label{eqn:prob_measure}
\end{equation}
for all $(c_1,c_2,\dots,c_n)\in\{-1,1\}^n$.
\begin{theorem}
\label{thm:conv_almost_surely}
Let $(a_1,a_2,\dots)$ be drawn from $\Omega$, equipped with the probability measure defined by~\eqref{eqn:prob_measure}, and write $A_n=(a_1,a_2,\dots,a_n)$. Let $r\ge 2$ be a fixed integer. Then, as $n\to\infty$,
\[
\frac{C_r(A_n)}{\sqrt{2n\log \binom{n}{r-1}}}\to 1\quad\text{almost surely}.
\]
\end{theorem}
\par
Alon, Kohayakawa, Mauduit, Moreira, and R\"odl~\cite{Alon2007} also proved a result on the asymptotic order of $C_r(A_n)$ that holds uniformly for a large range of~$r$.
\begin{abctheorem}[{\cite[Theorem~2]{Alon2007}}]
\label{thma}
Let $A_n$ be drawn uniformly at random from $\{-1,1\}^n$. Then the probability that
\[
\frac{2}{5}\sqrt{n\log\binom{n}{r}}<C_r(A_n)<\sqrt{\bigg(2+\frac{\log\log n}{\log n}\bigg)n\log \bigg(n\binom{n}{r}\bigg)}
\]
holds for all $r$ satisfying $2\le r\le n/4$ tends to $1$ as $n\to\infty$.
\end{abctheorem}
\par
We improve the upper bound in Theorem~\ref{thma} as follows. 
\par
\begin{theorem}
\label{thm:upper_bound}
Let $A_n$ be drawn uniformly at random from $\{-1,1\}^n$ and let $\e>0$ be real. Then, as $n\to\infty$,
\[
\Pr\Big[C_r(A_n)\le(1+\e)\sqrt{2n\log \tbinom{n}{r-1}}\quad\text{for all $r$ satisfying $2\le r\le n$}\Big]\to 1.
\]
\end{theorem}
\par
In view of Theorem~\ref{thm:conv_almost_surely}, the bound in Theorem~\ref{thm:upper_bound} is essentially best possible. We also note that Theorem~\ref{thm:upper_bound} gives the currently strongest existence result. (The computation of the asymptotic behaviour of the correlation measures of individual binary sequences is a notoriously difficult problem and, in the light of Theorem~\ref{thm:conv_almost_surely}, the currently known results tend to be unsatisfying, see for example~\cite[Theorem~1]{Mauduit1997}.)
\par
We shall prove Theorem~\ref{thm:upper_bound} in Section~\ref{sec:typical_upper_bound}. In Section~\ref{sec:expected_value}, we shall determine the limit of the expected value of~\eqref{eqn:_question_limit_distribution} (Proposition~\ref{pro:ECr}). We shall then use this result in Section~\ref{sec:almost_sure_conv} to deduce Theorem~\ref{thm:conv_almost_surely}.
\par
We now turn to lower bounds for $C_r(A_n)$. It is known that
\[
\min_{A_n\in\{-1,1\}^n}C_r(A_n)=1\quad\text{for odd $r$},
\]
which arises from the alternating sequence $(1,-1,1,-1,\dots)$. Therefore, interesting results can only be expected for even $r$. Indeed the following result was established by Alon, Kohayakawa, Mauduit, Moreira, and R\"odl~\cite{Alon2006}.
\begin{abctheorem}[{\cite[Theorem~1.1]{Alon2006}}]
\label{thmb}
Let $r$ and $n$ be positive integers with $r\le n/2$. Then
\[
C_{2r}(A_n)>\sqrt{\frac{1}{2}\bigg\lfloor\frac{n}{2r+1}\bigg\rfloor}
\]
for all $A_n\in\{-1,1\}^n$.
\end{abctheorem}
\par
Theorem~\ref{thmb} gives an affirmative answer to a problem due to Cassaigne, Mauduit, and S\'ark\"ozy~\cite[Problem~2]{Cassaigne2002}, which was suspected to be `really difficult' in~\cite[p.~109]{Cassaigne2002}. While the proof of Theorem~\ref{thmb} in~\cite{Alon2006} is quite involved, we shall show that Theorem~\ref{thmb} is a simple consequence of the so-called Welch bound~\cite{Welch1974}. This bound is an elementary result on the maximum nontrivial scalar products over a set of vectors.
\par
We also establish, as another consequence of the Welch bound, the following result, which was proved in~\cite{Alon2006} without an explicit lower bound for~$c_k$.
\begin{theorem}
\label{thm:lower_bound_max}
There exists a sequence of real numbers $c_k$, satisfying $c_k>1/9$ for each $k\ge 3$ and $c_k\to1/\sqrt{6e}=0.2476\dots$ as $k\to\infty$, such that for all positive integers $s$ and $n$ with $s\le n/3$, we have
\[
\max\big\{C_2(A_n),C_4(A_n),\dots,C_{2s}(A_n)\big\}>c_n\sqrt{sn}
\]
for all $A_n\in\{-1,1\}^n$.
\end{theorem}
\par
Theorems~\ref{thmb} and~\ref{thm:lower_bound_max} will be proved in Section~\ref{sec:minimal}.


\section{Typical upper bound}
\label{sec:typical_upper_bound}

In this section, we shall prove Theorem~\ref{thm:upper_bound}. The key ingredient in the proof will be an estimate for the range of a random walk. Let $X_1,\dots,X_n$ be independent random variables, each taking the values $-1$ or~$1$, each with probability $1/2$. Define the random variable
\begin{equation}
R_n=\max_{1\le m_1\le m_2\le n}\Biggabs{\sum_{j=m_1}^{m_2}X_j},   \label{eqn:range}
\end{equation}
which is called the \emph{range} of the random walk with steps $X_1,X_2,\dots$.
\par
We begin with a minor generalisation of a lemma due to Aistleitner~\cite[Lemma~2.3]{Aistleitner2013}.
\begin{lemma}
\label{lem:max_rw_dyadic}
Let $p$ be a nonnegative integer and let $n$ be an integer of the form
\[
\text{$j2^m$, where $j,m\in\Z$, $2^p<j\le 2^{p+1}$, and $m\ge 1$}.
\]
Then, for $\lambda>2\sqrt{n}$,
\[
\Pr\Big[R_n>\lambda(1+12\cdot 2^{-p/2})\Big]\le 2^{2p+4}\exp\bigg(-\frac{\lambda^2}{2n}\bigg).
\]
\end{lemma}
\par
Aistleitner's lemma~\cite[Lemma~2.3]{Aistleitner2013} is obtained by setting $p=10$ in Lemma~\ref{lem:max_rw_dyadic}. The general version can be proved by applying obvious modifications to the proof of~\cite[Lemma~2.3]{Aistleitner2013}, which is proved using a dyadic decomposition technique. (Aistleitner's lemma has the additional assumption that $n$ is sufficiently large, which however is not required in the proof.)
\par
We now proceed similarly as in~\cite{Aistleitner2013} and prove the following lemma, which holds for general $n$.
\begin{lemma}
\label{lem:large_deviation_max}
Let $\delta>0$ be real. Then, there exists a constant $n_0=n_0(\delta)$, such that for all $n\ge n_0$ and all $\lambda>2\sqrt{n}$,
\[
\Pr\Big[R_n>\lambda(1+\delta)\Big]\le (\log n)\exp\bigg(-\frac{\lambda^2}{2n}\bigg).
\]
\end{lemma}
\begin{proof}
Let $p$ be a positive integer and let $\hat n$ be the smallest integer that satisfies $\hat n\ge n$ and is of the form
\[
\text{$j2^m$, where $j,m\in\Z$, $2^p<j\le 2^{p+1}$, and $m\ge 1$}.
\]
We readily verify that
\begin{equation}
\frac{\hat n}{n}\le 1+\frac{1}{2^p}\quad\text{for $n\ge 2^{p+1}$}.   \label{eqn:hatn_n}
\end{equation}
Let $n\ge 2^{p+1}$ and $\lambda>2\sqrt{n}$, so that $\lambda\sqrt{1+2^{-p}}>2\sqrt{\hat n}$. Then
\begin{align*}
\Pr\Big[R_n>\lambda(1+12\cdot 2^{-p/2})\sqrt{1+2^{-p}}\Big]
\le &\Pr\Big[R_{\hat n}>\lambda(1+12\cdot 2^{-p/2})\sqrt{1+2^{-p}}\Big]\\
\le & \;2^{2p+4}\exp\bigg(-\frac{\lambda^2(1+2^{-p})}{2\hat n}\bigg)\\
\le & \;2^{2p+4}\exp\bigg(-\frac{\lambda^2}{2n}\bigg),
\end{align*}
by Lemma~\ref{lem:max_rw_dyadic} and~\eqref{eqn:hatn_n}. For $n>2$, we take $p=p(n)=\lfloor\frac{1}{2}\log\log n\rfloor$, so that $n\ge 2^{p+1}$. Moreover
\[
(1+12\cdot 2^{-p/2})\sqrt{1+2^{-p}}\le 1+\delta
\]
and $2^{2p+4}\le \log n$ for all $n\ge n_0$, where $n_0$ depends only on $\delta$. This completes the proof.
\end{proof}
\par
Before proving Theorem~\ref{thm:upper_bound}, we record the following elementary, albeit very useful, fact. 
\begin{lemma}
\label{lem:independence}
Let $X_1,X_2,\dots,X_n$ be mutually independent random variables, each taking each of the values $-1$ and $1$ with probability $1/2$ and let $u_1,\dots,u_r$ be integers satisfying
\[
0\le u_1<u_2<\cdots<u_r<n. 
\]
Then the $n-u_r$ products
\[
X_{1+u_1}X_{1+u_2}\cdots X_{1+u_r},\dots,X_{n-u_r+u_1}X_{n-u_r+u_2}\cdots X_n
\]
are mutually independent.
\end{lemma}
\par
For $r=2$, a formal proof of Lemma~\ref{lem:independence} is provided by Mercer~\cite[Proposition~1.1]{Mercer2006}.
\par
We now give a proof of Theorem~\ref{thm:upper_bound}. In this proof and in the remainder of this paper we make repeated use of the elementary bound
\begin{equation}
\bigg(\frac{n}{k}\bigg)^k\le\binom{n}{k}\le\bigg(\frac{en}{k}\bigg)^k \quad\text{for $k,n\in\Z$ satisfying $1\le k\le n$}.   \label{eqn:binomial_bound}
\end{equation}
\begin{proof}[Proof of Theorem~\ref{thm:upper_bound}]
Write $A_n=(a_1,a_2,\dots,a_n)$ and notice that $C_r(A_n)$ can be rewritten as
\begin{equation}
C_r(A_n)=\max_{0<u_2<\cdots<u_r<n}\;\max_{1\le m_1\le m_2\le n-u_r}\Biggabs{\sum_{j=m_1}^{m_2}a_ja_{j+u_2}\cdots a_{j+u_r}}.   \label{eqn:C_r_rewritten}
\end{equation}
Let $r$ be an integer satisfying $2\le r\le n$ and let $u_2,u_3,\dots,u_r$ be integers satisfying
\begin{equation}
0<u_2<\cdots<u_r<n.   \label{eqn:u_tuples}
\end{equation}
Write
\[
\lambda=\sqrt{2n\log \tbinom{n}{r-1}}.
\]
Then, in view of Lemma~\ref{lem:independence}, the probability
\begin{equation}
\Pr\Bigg[\max_{1\le m_1\le m_2\le n-u_r}\Biggabs{\sum_{j=m_1}^{m_2}a_ja_{j+u_2}\cdots a_{j+u_r}}>\lambda(1+\e)\Bigg]   \label{eqn:prob_max_corr}
\end{equation}
is at most $\Pr[R_n>\lambda(1+\e)]$ with $R_n$ defined as in~\eqref{eqn:range}. Write $1+\e=\sqrt{1+\gamma}(1+\delta)$ for some $\gamma,\delta>0$. By Lemma~\ref{lem:large_deviation_max}, there is a constant $n_0$, depending only on $\delta$, such that for all $n\ge n_0$, the probability~\eqref{eqn:prob_max_corr} is at most
\[
(\log n)\exp\bigg(-\frac{\lambda^2(1+\gamma)}{2n}\bigg)=\frac{\log n}{\binom{n}{r-1}^{1+\gamma}}.
\]
Summing over all possible tuples $(u_2,u_3,\dots,u_r)$ satisfying~\eqref{eqn:u_tuples}, we see from~\eqref{eqn:C_r_rewritten} that, for all $n\ge n_0$,
\begin{align}
\Pr\big[C_r(A_n)>\lambda(1+\e)\big]&\le \frac{(\log n)\binom{n-1}{r-1}}{\binom{n}{r-1}^{1+\gamma}}   \nonumber\\
&<\frac{\log n}{\binom{n}{r-1}^\gamma}.   \label{eqn:Pr_Cr_bound}
\end{align}
To prove the theorem, it is enough to show that, as $n\to\infty$,
\[
\sum_{r=2}^n\Pr\big[C_r(A_n)>\lambda(1+\e)\big]\to 0.
\]
From~\eqref{eqn:Pr_Cr_bound}, for $n\ge n_0$, the left hand side is at most
\[
\sum_{k=1}^{n-1}\frac{\log n}{\binom{n}{k}^\gamma}.
\]
Let $m$ be an integer such that $m\gamma>1$. Then, for $n\ge m$, this last expression is at most
\begin{align*}
2\sum_{k=1}^{m-1}\frac{\log n}{\binom{n}{k}^\gamma}+2\sum_{k=m}^{\lfloor n/2\rfloor}\frac{\log n}{\binom{n}{k}^\gamma}&\le \frac{2m\log n}{n^\gamma}+\frac{n\log n}{\binom{n}{m}^\gamma}\\
&\le\frac{2m\log n}{n^\gamma}+\frac{m^{m\gamma}\log n}{n^{m\gamma-1}},
\end{align*}
using~\eqref{eqn:binomial_bound}. Since $\gamma>0$ and $m\gamma>1$, the right hand side tends to zero as $n\to\infty$, as required.
\end{proof}


\section{Asymptotic expected value}
\label{sec:expected_value}

In this section, we prove the following result, which is a key step in the proof of Theorem~\ref{thm:conv_almost_surely}.
\begin{proposition}
\label{pro:ECr}
Let $A_n$ be drawn uniformly at random from $\{-1,1\}^n$. Then, as $n\to\infty$,
\[
\frac{\E\big[C_r(A_n)\big]}{\sqrt{2n\log \binom{n}{r-1}}}\to 1.
\]
\end{proposition}
\par
To prove this proposition, we make repeated use of the following lemma, which follows from well known results on concentration of probability measures (see McDiarmid~\cite{McDiarmid1989}, for example).
\begin{lemma}[{\cite[Inequality~(99)]{Alon2007}}]
\label{lem:concentration}
Let $A_n$ be drawn uniformly at random from $\{-1,1\}^n$. Then, for $\theta\ge 0$,
\[
\Pr\Big[\bigabs{C_r(A_n)-\E[C_r(A_n)]}\ge\theta\Big]\le 2\exp\bigg(-\frac{\theta^2}{2r^2n}\bigg).
\]
\end{lemma}
\par 
By combining Lemma~\ref{lem:concentration} and Theorem~\ref{thm:upper_bound}, it is readily verified that
\begin{equation}
\limsup_{n\to\infty}\frac{\E\big[C_r(A_n)\big]}{\sqrt{2n\log \binom{n}{r-1}}}\le 1.   \label{eqn:E_upper_bound}
\end{equation}
In studying a problem that is related to the second order correlation measure of finite binary sequences, the author proved in~\cite{Schmidt2014} that
\[
\liminf_{n\to\infty}\frac{\E\big[C_2(A_n)\big]}{\sqrt{2n\log n}}\ge 1,
\]
which proves Proposition~\ref{pro:ECr} for $r=2$. Our proof of the general case is also based on the approach of~\cite{Schmidt2014}. 
\par
Let $A_n=(a_1,a_2\dots,a_n)$ be an element of $\{-1,1\}^n$ and, for integers $u_2,\dots,u_r$ satisfying
\[
0<u_2<u_3<\cdots<u_r<n,
\]
define 
\[
S_{u_2,\dots,u_r}(A_n)=\sum_{j=1}^{n-u_r}a_ja_{j+u_2}\cdots a_{j+u_r}.
\]
The key ingredients to the proof of Proposition~\ref{pro:ECr} are the following two lemmas on $S_{u_2,\dots,u_r}(A_n)$, which generalise~\cite[Proposition~2.1]{Schmidt2014} and~\cite[Proposition~2.7]{Schmidt2014}, respectively, from $r=2$ to general $r\ge 2$. These lemmas can be proved by modifying the arguments used in~\cite{Schmidt2014}. As the modifications are not always obvious, we include proofs at the end of this section.
\begin{lemma}
\label{lem:Pr_Cu}
Let $A_n$ be drawn uniformly at random from $\{-1,1\}^n$ and let $r\ge 2$ be an integer. Then there exists a constant $n_0=n_0(r)$, such that for all $n\ge n_0$ and all
\[
0<u_2<u_3<\cdots<u_r\le \frac{n}{\log n},
\]
we have
\begin{equation}
\Pr\Big[\abs{S_{u_2,\dots,u_r}(A_n)}\ge \sqrt{2n\log \tbinom{n}{r-1}}\Big]\ge\frac{1}{5e^{r-2}\tbinom{n}{r-1}\sqrt{\log \tbinom{n}{r-1}}}.   \label{eqn:prob_Cu}
\end{equation}
\end{lemma}
\par
\begin{lemma}
\label{lem:Pr_CuCv}
Let $A_n$ be drawn uniformly at random from $\{-1,1\}^n$, let $r\ge 2$ be an integer, and write
\[
\lambda=\sqrt{2n\log \tbinom{n}{r-1}}.
\]
Let $u_2<u_3<\cdots<u_r$ and $v_2<v_3<\cdots<v_r$ be positive integers strictly less than $n$ satisfying $(u_2,\dots,u_r)\ne(v_2,\dots,v_r)$. Then there exists a constant $n_0=n_0(r)$, such that for all $n\ge n_0$, we have
\begin{equation}
\Pr\big[\abs{S_{u_2,\dots,u_r}(A_n)}\ge \lambda\,\cap\,\abs{S_{v_2,\dots,v_r}(A_n)}\ge \lambda\big]\le \frac{23}{\tbinom{n}{r-1}^2}.   \label{eqn:prob_upper_bound}
\end{equation}
\end{lemma}
\par
We now prove Proposition~\ref{pro:ECr}.
\par
\begin{proof}[Proof of Proposition~\ref{pro:ECr}]
Let $\delta>0$ and define the set
\begin{equation}
N(\delta)=\bigg\{n\ge r:\frac{\E\big[C_r(A_n)\big]}{\sqrt{2n\log \binom{n}{r-1}}}<1-\delta\bigg\}.   \label{eqn:def_N_delta}
\end{equation}
We shall show that $N(\delta)$ has finite size for all choices of $\delta>0$, which together with~\eqref{eqn:E_upper_bound} proves the proposition. To do so, we define the set
\[
W=\bigg\{(u_2,u_3,\dots,u_r)\in\Z^{r-1}:0<u_2<u_3<\cdots<u_r\le \frac{n}{\log n}\bigg\}.
\]
Since
\[
C_r(A_n)\ge\max_{(u_2,\dots,u_r)\in W}\,\bigabs{S_{u_2,\dots,u_r}(A_n)},
\]
we find by the inclusion-exclusion principle that, for all real $\lambda$,
\begin{multline*}
\Pr\big[C_r(A_n)\ge \lambda\big]\ge\sum_{(u_2,\dots,u_r)\in W}\Pr\big[\abs{S_{u_2,\dots,u_r}(A_n)}\ge \lambda\big]\\[1ex]
\quad-\frac{1}{2}\sums{(u_2,\dots,u_r),(v_2,\dots,v_r)\in W\\(u_2,\dots,u_r)\ne(v_2,\dots,v_r)}\Pr\big[\abs{S_{u_2,\dots,u_r}(A_n)}\ge \lambda\cap\abs{S_{v_2,\dots,v_r}(A_n)}\ge \lambda\big].
\end{multline*}
Now take
\begin{equation}
\lambda=\sqrt{2n\log \tbinom{n}{r-1}}   \label{eqn:lambda}
\end{equation}
and apply Lemmas~\ref{lem:Pr_Cu} and~\ref{lem:Pr_CuCv} to get, for all sufficiently large $n$,
\begin{equation}
\Pr\big[C_r(A_n)\ge \lambda\big]\ge \abs{W}\cdot\frac{1}{5e^{r-2}\tbinom{n}{r-1}\sqrt{\log \tbinom{n}{r-1}}}-\frac{\abs{W}^2}{2}\cdot\frac{23}{\tbinom{n}{r-1}^2}.   \label{eqn:Pr_Cr_W}
\end{equation}
We have
\[
\abs{W}=\binom{\lfloor n/\log n\rfloor}{r-1}
\]
and by the elementary bounds~\eqref{eqn:binomial_bound} for binomial coefficients we find that, for all sufficiently large $n$,
\begin{align*}
\abs{W}&\le \bigg(\frac{en}{(r-1)\log n}\bigg)^{r-1}\le\bigg(\frac{e}{\log n}\bigg)^{r-1}\binom{n}{r-1}
\intertext{and}
\abs{W}&\ge \bigg(\frac{n}{2(r-1)\log n}\bigg)^{r-1}\ge \bigg(\frac{1}{2e\log n}\bigg)^{r-1}\binom{n}{r-1}.
\end{align*}
Hence, from~\eqref{eqn:Pr_Cr_W} we obtain, for all sufficiently large $n$,
\[
\Pr\big[C_r(A_n)\ge \lambda\big]\ge \frac{1}{5e^{r-2}}\bigg(\frac{1}{2e\log n}\bigg)^{r-1}\frac{1}{\sqrt{r\log n}}-12\bigg(\frac{e}{\log n}\bigg)^{2r-2}.
\]
Since $r\ge 2$, the first term on the right hand side dominates, and so a crude estimate gives
\begin{equation}
\Pr\big[C_r(A_n)\ge \lambda\big]\ge \frac{1}{e^{3r}\sqrt{r}}\,\bigg(\frac{1}{\log n}\bigg)^{r-1/2}   \label{eqn:Pr_C_lb}
\end{equation}
for all sufficiently large $n$. By the definition~\eqref{eqn:def_N_delta} of $N(\delta)$, we have $\lambda>\E[C_r(A_n)]$ for all $n\in N(\delta)$, and thus find from Lemma~\ref{lem:concentration} with $\theta=\lambda-\E[C_r(A_n)]$ that, for all $n\in N(\delta)$,
\[
\Pr\big[C_r(A_n)\ge \lambda\big]\le 2\exp\bigg(-\frac{(\lambda-\E[C_r(A_n)])^2}{2r^2n}\bigg).
\]
Comparison with~\eqref{eqn:Pr_C_lb} then gives, for all sufficiently large $n\in N(\delta)$,
\[
\frac{1}{e^{3r}\sqrt{r}}\,\bigg(\frac{1}{\log n}\bigg)^{r-1/2}\le 2\exp\bigg(-\frac{(\lambda-\E[C_r(A_n)])^2}{2r^2n}\bigg),
\]
or equivalently, after substituting the value~\eqref{eqn:lambda} for $\lambda$,
\[
\frac{\E[C_r(A_n)]}{\sqrt{2n\log \binom{n}{r-1}}}\ge 1-\sqrt{\frac{r^2(r-1/2)\log\log n+r^2\log (2e^{3r}\sqrt{r})}{\log\binom{n}{r-1}}}.
\]
Hence, by the definition~\eqref{eqn:def_N_delta} of $N(\delta)$, we see that $N(\delta)$ has finite size for all choices of $\delta>0$, as required.
\end{proof}
\par
In the remainder of this section, we provide proofs of Lemmas~\ref{lem:Pr_Cu} and~\ref{lem:Pr_CuCv}.
\begin{proof}[Proof of Lemma~\ref{lem:Pr_Cu}]
We adopt the standard notation $x_n\sim y_n$ to mean that $x_n=y_n(1+o(1))$ as $n\to\infty$. By Lemma~\ref{lem:independence}, $S_{u_2,\dots,u_r}(A_n)$ is a sum of $n-u_r$ mutually independent random variables, each taking each of the values $-1$ and $+1$ with probability $1/2$. We use a normal approximation to estimate the tail of the distribution of $\abs{S_{u_2,\dots,u_r}(A_n)}$ (see Feller~\cite[Chapter VII, (6.7)]{Feller1968}, for example): If $\xi_n\to\infty$ in such a way that $\xi_n^3/\sqrt{n}\to 0$ as $n\to\infty$, then 
\[
\Pr\Big[\bigabs{S_{u_2,\dots,u_r}(A_n)}\ge\xi_n\sqrt{n-u_r}\Big]\sim \sqrt{\frac{2}{\pi}}\cdot\frac{1}{\xi_n}\,\exp\bigg(-\frac{\xi_n^2}{2}\bigg).
\]
Taking
\[
\xi_n=\sqrt{\frac{2n}{n-u_r}\,\log\binom{n}{r-1}}
\]
gives, since $\frac{n}{n-u_r}\sim 1$,
\begin{multline}
\Pr\Big[\bigabs{S_{u_2,\dots,u_r}(A_n)}\ge\sqrt{2n\log \tbinom{n}{r-1}}\Big]\\
\sim \frac{1}{\sqrt{\pi\log \binom{n}{r-1}}}\,\exp\bigg(-\frac{n}{n-u_r}\log \binom{n}{r-1}\bigg).   \label{eqn:Pr_Cu_F}
\end{multline}
Using $u_r\le \frac{n}{\log n}$, we have
\[
\exp\bigg(-\frac{n}{n-u_r}\log \binom{n}{r-1}\bigg)\ge \exp\bigg(-\frac{\log n}{\log n-1}\,\log \binom{n}{r-1}\bigg),
\]
and then, since
\[
\exp\bigg(-\frac{\log n}{\log n-1}\,\log \binom{n}{r-1}\bigg)\sim\frac{1}{e^{r-1}\binom{n}{r-1}}
\]
and $e\sqrt{\pi}<5$, we find from~\eqref{eqn:Pr_Cu_F} that
\[
\Pr\Big[\bigabs{S_{u_2,\dots,u_r}(A_n)}\ge\sqrt{2n\log \tbinom{n}{r-1}}\Big]\ge\frac{1}{5e^{r-2}\binom{n}{r-1}\sqrt{\log \binom{n}{r-1}}}
\]
for all sufficiently large~$n$.
\end{proof}
\par
To prove Lemma~\ref{lem:Pr_CuCv}, it is convenient to use the following notation.
\begin{definition}
A tuple $(x_1,\dots,x_{2m})$ is \emph{$d$-even} if there exists a permutation $\sigma$ of $\{1,2,\dots,2m\}$ such that $x_{\sigma(2i-1)}=x_{\sigma(2i)}$ for each $i\in\{1,2,\dots,d\}$ and $d$ is the largest integer with this property. An $m$-even tuple $(x_1,\dots,x_{2m})$ is just called \emph{even}.
\end{definition}
\par
For example, $(1,3,1,4,3,4)$ is even, while $(2,1,1,2,1,3)$ is $2$-even. In the next two lemmas we state two results about even tuples.
\par
Recall that, for a positive integer $k$, the double factorial
\[
(2k-1)!!=\frac{(2k)!}{k!\,2^k}=(2k-1)(2k-3)\cdots 3\cdot 1
\]
is the number of ways to arrange $2k$ objects into $k$ unordered pairs. The following lemma is immediate.
\begin{lemma}[{\cite[Lemma~2.4]{Schmidt2014}}]
\label{lem:even_tuple_single}
Let $m$ and $q$ be positive integers. Then the number of even tuples in $\{1,\dots,m\}^{2q}$ is at most $(2q-1)!!\,m^q$.
\end{lemma}
\par
The following lemma generalises~\cite[Lemma~2.5]{Schmidt2014}.
\begin{lemma}
\label{lem:even_tuple_double}
Let $n$, $q$, and $t$ be positive integers satisfying $0\le t<q$ and let $u_2<u_3<\dots<u_r$ and $v_2<v_3<\dots<v_r$ be positive integers strictly less than $n$ satisfying $(u_2,\dots,u_r)\ne (v_2,\dots,v_r)$. Write $I=\{1,\dots,2q\}$ and let $S$ be the subset of $\{1,\dots,n\}^{4rq}$ containing all even elements
\[
(x_i,x_i+u_2,\dots,x_i+u_r,y_i,y_i+v_2,\dots,y_i+v_r)_{i\in I}
\]
such that $(x_i)_{i\in I}$ is $d$-even for some $d<q-t$. Then
\[
\abs{S}\le (4rq-1)!!\,n^{2q-(t+1)/3}.
\]
\end{lemma}
\begin{proof}
We will construct a set of tuples that contains $S$ as a subset. For convenience write $u_1=v_1=0$.  Arrange the $4rq$ variables
\begin{equation}
x_i+u_k,\,y_i+v_k\quad\text{for $i\in I$ and $k\in\{1,2,\dots,r\}$}   \label{eqn:xy}
\end{equation}
into $2rq$ unordered pairs $\{a_1,b_1\},\{a_2,b_2\},\dots,\{a_{2rq},b_{2rq}\}$ such that there are at most $q-t-1$ pairs $\{x_i,x_j\}$. This can be done in at most $(4rq-1)!!$ ways. We formally set $a_i=b_i$ for all $i\in\{1,2,\dots,2rq\}$. If this assignment does not yield a contradiction, then we call the arrangement of~\eqref{eqn:xy} into $2rq$ pairs \emph{consistent}. For example, if there are pairs of the form $\{x_i,y_j\}$, $\{x_i+u_2,y_j+v_2\}$, \dots, $\{x_i+u_r,y_j+v_r\}$, then the arrangement is not consistent since $(u_2,\dots,u_r)\ne (v_2,\dots,v_r)$ by assumption. 
\par
Notice that, if there is a pair of the form $\{x_i+u_k,x_j+u_\ell\}$ in a consistent arrangement, then $i\ne j$ and~$x_i$ determines~$x_j$. Likewise, if there is a pair of the form $\{y_i+v_k,y_j+v_\ell\}$ in a consistent arrangement, then $i\ne j$ and~$y_i$ determines~$y_j$. On the other hand, if a consistent arrangement contains a pair of the form $\{x_i+u_k,y_j+v_\ell\}$, then $x_i$ determines $y_j$ and at least one other variable in the list
\begin{equation}
x_1,\dots,x_{2q},y_1,\dots,y_{2q}.   \label{eqn:var_list}
\end{equation}
To see this, it is enough to show that a consistent arrangement cannot contain $r$ pairs involving only the $2r$ variables
\begin{equation}
x_i+u_1,\dots,x_i+u_r,y_j+v_1,\dots,y_j+v_r.   \label{eqn:vars}
\end{equation}
Indeed, since $u_1<\dots<u_r$ and $v_1<\dots<v_r$ and $u_1=v_1=0$, the only possibility for such $r$ pairs would be $\{x_i+u_1,y_j+v_1\}$, \dots, $\{x_i+u_r,y_j+v_r\}$. However, as already mentioned above, this implies that the arrangement is not consistent. Hence at least one of the variables in the list~\eqref{eqn:vars} must be paired with a variable not in the list~\eqref{eqn:vars}, and so $x_i$ determines another variable in the list~\eqref{eqn:var_list} different from~$y_j$.
\par
Now, by assumption, each consistent arrangement contains at most $q-t-1$ pairs of the form $\{x_i,x_j\}$ and at most $q$ pairs of the form $\{y_i,y_j\}$, and so at most 
\[
q-t-1+q+\tfrac{1}{3}(2t+2)=2q-\tfrac{1}{3}(t+1)
\]
of the variables in~\eqref{eqn:var_list} can be chosen independently. We assign to each of these a value of $\{1,\dots,n\}$. In this way, we construct a set of at most $(4rq-1)!!\,n^{2q-(t+1)/3}$ tuples that contains $S$ as a subset.
\end{proof}
\par
The next lemma, whose proof is modelled on that of~\cite[Lemma~2.6]{Schmidt2014}, provides the key step in the proof of Lemma~\ref{lem:Pr_CuCv}.
\begin{lemma}
\label{lem:moments}
Let $p$ and $h$ be integers satisfying $0\le h<p$ and let $A_n$ be drawn uniformly at random from $\{-1,1\}^n$. Let $u_2<u_3<\cdots<u_r$ and $v_2<v_3<\cdots<v_r$ be positive integers strictly less than $n$ satisfying $(u_2,\dots,u_r)\ne(v_2,\dots,v_r)$. Then
\begin{multline}
\E\Big[\big(S_{u_2,\dots,u_r}(A_n)S_{v_2,\dots,v_r}(A_n)\big)^{2p}\Big]\\
\le n^{2p}\big[(2p-1)!!\big]^2\bigg(1+\frac{(4rp)^{4rh}}{n^{1/3}}+\frac{(4rp)^{2rp}}{n^{(h+1)/3}}\bigg).   \label{eqn:moment}
\end{multline}
\end{lemma}
\begin{proof}
Write $A_n=(a_1,a_2,\dots,a_n)$. Expand to see that the left hand side of~\eqref{eqn:moment} equals
\begin{multline}
\sum_{i_1,\dots,i_{2p}=1}^{n-u_r}\;\sum_{j_1,\dots,j_{2p}=1}^{n-v_r}\E\big[a_{i_1}a_{i_1+u_2}\cdots a_{i_1+u_r} \cdots a_{i_{2p}}a_{i_{2p}+u_2}\cdots a_{i_{2p}+u_r}\\
a_{j_1}a_{j_1+v_2}\cdots a_{j_1+v_r}\cdots a_{j_{2p}}a_{j_{2p}+v_2}\cdots a_{j_{2p}+v_r}\big].   \label{eqn:moment_expanded}
\end{multline}
Write $I=\{1,2,\dots,2p\}$ and let $T$ be the set containing all even tuples in $\{1,\dots,n\}^{4rp}$ of the form
\begin{equation}
(x_i,x_i+u_2,\dots,x_i+u_r,y_i,y_i+v_2,\dots,y_i+v_r)_{i\in I}.   \label{eqn:element_of_T}
\end{equation}
Since $a_1,\dots,a_n$ are mutually independent, $\E[a_j]=0$, and $a_j^2=1$ for all $j\in\{1,\dots,n\}$, we find from~\eqref{eqn:moment_expanded} that the left hand side of~\eqref{eqn:moment} equals $\abs{T}$. It remains to show that $\abs{T}$ is at most the right hand side of~\eqref{eqn:moment}.
\par
We define the following subsets of $T$.
\begin{list}{$\bullet$}{\leftmargin=1em \itemindent=0em \itemsep=0ex}
\item $T_1$ contains all elements~\eqref{eqn:element_of_T} of $T$ such that $(x_i)_{i\in I}$ and $(y_i)_{i\in I}$ are even.

\item $T_2$ contains all elements~\eqref{eqn:element_of_T} of $T$ such that $(x_i)_{i\in I}$ is $d_1$-even and $(y_i)_{i\in I}$ is $d_2$-even for some $d_1$ and $d_2$ satisfying $p-h\le d_1,d_2\le p$, at least one of them strictly less than $p$.

\item $T_3$ contains all elements~\eqref{eqn:element_of_T} of $T$ such that $(x_i)_{i\in I}$ or $(y_i)_{i\in I}$ is $d$-even for some $d<p-h$.
\end{list}
It is readily verified that $T_1$, $T_2$, and $T_3$ partition $T$. We now bound the cardinalities of $T_1$, $T_2$, and $T_3$.
\par
{\itshape The set $T_1$.} Using Lemma~\ref{lem:even_tuple_single} applied with $q=p$, we have
\begin{equation}
\abs{T_1}\le \big[(2p-1)!!\big]^2\,n^{2p}.   \label{eqn:T1}
\end{equation}
\par
{\itshape The set $T_2$.} 
Consider an element~\eqref{eqn:element_of_T} of $T_2$. Then there exist $(2p-2h)$-element subsets $J$ and $K$ of $I$ such that $(x_i)_{i\in J}$ and $(y_i)_{i\in K}$ are even and
\begin{equation}
(x_i)_{i\in I\setminus J}   \label{eqn:tuple_2h_lr}
\end{equation}
is not even (if $(x_i)_{i\in I\setminus J}$ were even, then $(y_i)_{i\in I\setminus K}$ would also be even, which contradicts the definition of the elements of $T_2$). Since $(x_i)_{i\in J}$ and $(y_i)_{i\in K}$ are even and the tuple~\eqref{eqn:element_of_T} is even, we find that
\begin{equation}
(x_i,x_i+u_2,\dots,x_i+u_r,y_j,y_j+v_2,\dots,y_j+v_r)_{i\in I\setminus J,\,j\in I\setminus K}   \label{eqn:tuple_2h}
\end{equation}
is also even. There are ${2p\choose 2h}$ subsets $J$ and ${2p\choose 2h}$ subsets $K$. By Lemma~\ref{lem:even_tuple_single} applied with $q=p-h$, for each such $J$ and $K$, there are at most $(2p-2h-1)!!\,n^{p-h}$ even tuples $(x_i)_{i\in J}$ satisfying $1\le x_i\le n$ for each $i\in J$ and at most $(2p-2h-1)!!\,n^{p-h}$ even tuples $(y_i)_{i\in K}$ satisfying $1\le y_i\le n$ for each $i\in K$. By Lemma~\ref{lem:even_tuple_double} applied with $q=h$ and $t=0$, the number of even tuples in $\{1,\dots,n\}^{4rh}$ of the form~\eqref{eqn:tuple_2h} such that the tuple in~\eqref{eqn:tuple_2h_lr} is not even is at most $(4rh-1)!!\,n^{2h-1/3}$. Therefore,
\begin{align}
\abs{T_2}&\le (4rh-1)!!\,n^{2h-1/3} \, \Big[\tbinom{2p}{2h}(2p-2h-1)!!\,n^{p-h}\Big]^2 \nonumber\\
&\le n^{2p-1/3}\big[(2p-1)!!\big]^2\,(4rp)^{4rh}.   \label{eqn:T2}   
\end{align}
\par
{\itshape The set $T_3$.} By Lemma~\ref{lem:even_tuple_double} applied with $q=p$ and $t=h$ and by symmetry, we have
\begin{equation}
\abs{T_3}\le2(4rp-1)!!\,n^{2p-(h+1)/3}\le n^{2p-(h+1)/3}(4rp)^{2rp}.   \label{eqn:T3}
\end{equation}
\par
Now from~\eqref{eqn:T1},~\eqref{eqn:T2}, and~\eqref{eqn:T3} we get an upper bound for $\abs{T}$, from which we can deduce~\eqref{eqn:moment}.
\end{proof}
\par
We now prove Lemma~\ref{lem:Pr_CuCv}.
\begin{proof}[Proof of Lemma~\ref{lem:Pr_CuCv}]
Let $X_1$ and $X_2$ be a random variables and let $p$ be a positive integer. Then by Markov's inequality, for $\theta_1,\theta_2>0$,
\[
\Pr\big[\abs{X_1}\ge\theta_1\,\cap\,\abs{X_2}\ge\theta_2\big]\le\frac{\E\big[(X_1X_2)^{2p}\big]}{(\theta_1\theta_2)^{2p}}.
\]
Let $h$ be an integer satisfying $0\le h<p$. Lemma~\ref{lem:moments} shows that the left hand side of~\eqref{eqn:prob_upper_bound} is at most
\begin{equation}
\frac{[(2p-1)!!\big]^2}{(2\log \tbinom{n}{r-1})^{2p}}\big[1+K_1(n,p,h)+K_2(n,p,h)\big],   
\label{Pr_moment}
\end{equation}
where
\begin{align*}
K_1(n,p,h)&=n^{-1/3}\,(4rp)^{4rh},\\
K_2(n,p,h)&=n^{-(h+1)/3}\,(4rp)^{2rp}.
\end{align*}
We take $p=\lfloor\log \tbinom{n}{r-1}\rfloor$ and $h=\lfloor \alpha\log\log n\rfloor$ for some $\alpha>0$, to be determined later, and show that~\eqref{Pr_moment} is at most $23/\tbinom{n}{r-1}^2$ for all sufficiently large~$n$. Notice that $h<p$ for all sufficiently large $n$, as assumed. By Stirling's approximation
\[
\sqrt{2\pi k}\;k^ke^{-k}\le k!\le\sqrt{3\pi k}\;k^ke^{-k},
\]
we have
\[
\frac{[(2p-1)!!\big]^2}{(2\log \tbinom{n}{r-1})^{2p}}\le \frac{3p^{2p}e^{-2p}}{(\log \tbinom{n}{r-1})^{2p}}\le\frac{3e^2}{\tbinom{n}{r-1}^2}.
\]
Moreover
\begin{align*}
K_1(n,p,h)&\le K_1(n,r\log n,\alpha\log\log n)\\
&=n^{-\frac{1}{3}}\,n^{\frac{2\alpha\log\log n(\log r+\log\log n)}{\log n}}\\
&=O(n^{-\frac{1}{4}})
\intertext{and}
K_2(n,p,h)&\le K_2(n,r\log n,(\alpha-1)\log\log n)\\
&=n^{-\frac{1}{3}-(\frac{\alpha-1}{3}-2r^2)\log\log n+2r^2\log (4r^2)}\\
&=O(n^{-\log\log n})
\end{align*}
by taking $\alpha=10r^2$, say. The lemma follows since $3e^2<23$.
\end{proof}


\section{Almost sure convergence}
\label{sec:almost_sure_conv}

In this section we prove Theorem~\ref{thm:conv_almost_surely}. We begin with the following standard result (see~\cite[Theorem~A.1.1]{Alon2008}, for example).
\begin{lemma}
\label{lem:large_deviation}
Let $X_1,\dots,X_n$ be independent random variables, each taking the values $-1$ and $1$, each with probability $1/2$. Then, for $\lambda\ge 0$,
\[
\Pr\Bigg[\Biggabs{\sum_{j=1}^nX_j}>\lambda\Bigg]\le 2\exp\bigg(-\frac{\lambda^2}{2n}\bigg).
\]
\end{lemma}
\par
Lemma~\ref{lem:large_deviation} is used to deduce the following result.
\begin{lemma}
\label{lem:prob_difference}
Let $(a_1,a_2,\dots)$ be drawn from $\Omega$, equipped with the probability measure defined by~\eqref{eqn:prob_measure}, and write $A_n=(a_1,a_2,\dots,a_n)$. Let $n_1,n_2,\dots$ be a strictly increasing sequence of integers greater than or equal to~$r$. Then, almost surely,
\[
C_r(A_{n_{k+1}})-C_r(A_{n_k})\le\sqrt{10(n_{k+1}-n_k)\log \tbinom{n_{k+1}}{r-1}}
\]
for all sufficiently large $k$.
\end{lemma}
\begin{proof}
Write
\begin{equation}
\lambda=\sqrt{10(n_{k+1}-n_k)\log\tbinom{n_{k+1}}{r-1}}.   \label{eqn:def_lambda}
\end{equation}
If \begin{equation}
C_r(A_{n_{k+1}})-C_r(A_{n_k})>\lambda,   \label{eqn:diff_C_r}
\end{equation}
then
\begin{equation}
\Biggabs{\sum_{j=\max(1,n_k-u_r+1)}^{m}a_{j+u_1}a_{j+u_2}\cdots a_{j+u_r}}>\lambda   \label{eqn:S_ge_lambda}
\end{equation}
for at least one tuple $(u_1,u_2,\dots,u_r)$ satisfying
\begin{equation}
0\le u_1<u_2<\cdots<u_r<n_{k+1}   \label{eqn:r_tuple_u}
\end{equation}
and at least one $m$ satisfying
\begin{equation}
n_k-u_r+1\le m\le n_{k+1}-u_r.   \label{eqn:k_range}
\end{equation}
Let $(u_1,u_2,\dots,u_r)$ be a tuple of integers satisfying~\eqref{eqn:r_tuple_u} and let $m$ be an integer satisfying~\eqref{eqn:k_range}. By Lemma~\ref{lem:independence}, the sum in~\eqref{eqn:S_ge_lambda} is a sum of at most $n_{k+1}-n_k$ independent random variables, each taking each of the values $1$ and $-1$ with probability~$1/2$. Thus, by Lemma~\ref{lem:large_deviation}, the probability of~\eqref{eqn:S_ge_lambda} is at most
\[
2\exp\bigg(-\frac{\lambda^2}{2(n_{k+1}-n_k)}\bigg)=2\binom{n_{k+1}}{r-1}^{-5},
\]
after substituting~\eqref{eqn:def_lambda}. Summing over all possible tuples $(u_1,u_2,\dots,u_r)$ and all possible $m$, the probability that~\eqref{eqn:S_ge_lambda} happens for some $(u_1,u_2,\dots,u_r)$ satisfying~\eqref{eqn:r_tuple_u} and some integer $m$ satisfying~\eqref{eqn:k_range} is at most
\[
2(n_{k+1}-n_k)\binom{n_{k+1}}{r}\binom{n_{k+1}}{r-1}^{-5}.
\]
This is also an upper bound for the probability of~\eqref{eqn:diff_C_r}, and so
\[
\Pr\big[C_r(A_{n_{k+1}})-C_r(A_{n_k})>\lambda\big]\le 2(n_{k+1})^2\binom{n_{k+1}}{r-1}^{-4}\le \frac{2}{(n_{k+1})^2}.
\]
Thus,
\[
\sum_{k=1}^\infty\Pr\big[C_r(A_{n_{k+1}})-C_r(A_{n_k})>\lambda\big]\le\sum_{k=1}^\infty\frac{2}{(n_{k+1})^2}<\infty,
\]
and the result follows from the Borel-Cantelli Lemma.
\end{proof}
\par 
We now prove Theorem~\ref{thm:conv_almost_surely}.
\begin{proof}[Proof of Theorem~\ref{thm:conv_almost_surely}]
Write
\[
\vartheta_n=\sqrt{2n\log \tbinom{n}{r-1}}
\]
and let $n_k$ be the smallest integer that is at least $e^{k^{1/2}}$. We first show that the theorem holds for the subsequence $n_k$, namely that, as $k\to\infty$,
\begin{equation}
\frac{C_r(A_{n_k})}{\vartheta_{n_k}}\to1\quad\text{almost surely}.   \label{eqn:almost_sure_conv_nk}
\end{equation}
To do so, choose an $\e>0$ and observe that by the triangle inequality, the probability
\[
\Pr\Bigg[\biggabs{\frac{C_r(A_n)}{\vartheta_n}-1}>\e\Bigg]
\]
is bounded from above by
\[
\Pr\Bigg[\biggabs{\frac{C_r(A_n)}{\vartheta_n}-\frac{\E[C_r(A_n)]}{\vartheta_n}}>\tfrac{1}{2}\e\Bigg]+\Pr\Bigg[\biggabs{\frac{\E[C_r(A_n)]}{\vartheta_n}-1}>\tfrac{1}{2}\e\Bigg].
\]
By Proposition~\ref{pro:ECr}, the second probability equals zero for all sufficiently large $n$. The first probability can be bounded using Lemma~\ref{lem:concentration}, showing that
\[
\Pr\Bigg[\biggabs{\frac{C_r(A_n)}{\vartheta_n}-1}>\tfrac{1}{2}\e\Bigg]\le 2\exp\bigg(-\frac{\e^2}{4r^2}\log\binom{n}{r-1}\bigg)
\]
for all sufficiently large $n$. We can further bound this expression very crudely by $1/(\log n)^3$, say, for all sufficiently large $n$. Thus, since $n_k\ge e^{k^{1/2}}$, we have for sufficiently large $k_0$,
\[
\sum_{k=k_0}^\infty\Pr\Bigg[\biggabs{\frac{C_r(A_{n_k})}{\vartheta_{n_k}}-1}>\tfrac{1}{2}\e\Bigg]\le\sum_{k=k_0}^\infty\frac{1}{(\log n_k)^3}\le \sum_{k=k_0}^\infty \frac{1}{k^{3/2}}<\infty
\]
and~\eqref{eqn:almost_sure_conv_nk} follows from the Borel-Cantelli Lemma.
\par
We shall now complete the proof by showing that, as $k\to\infty$,
\begin{equation}
\max_{n_k\le n\le n_{k+1}}\biggabs{\frac{C_r(A_n)}{\vartheta_n}-1}\to 0\quad\text{almost surely}.   \label{eqn:almost_sure_conv_max}
\end{equation}
We apply the triangle inequality to find that
\begin{multline}
\max_{n_k\le n\le n_{k+1}}\biggabs{\frac{C_r(A_n)}{\vartheta_n}-1}\le \biggabs{1-\frac{C_r(A_{n_{k+1}})}{\vartheta_{n_{k+1}}}}\\[1ex]
+\max_{n_k\le n\le n_{k+1}}\biggabs{\frac{C_r(A_{n_{k+1}})}{\vartheta_{n_{k+1}}}-\frac{C_r(A_n)}{\vartheta_{n_{k+1}}}} 
+\max_{n_k\le n\le n_{k+1}}\biggabs{\frac{C_r(A_n)}{\vartheta_{n_{k+1}}}-\frac{C_r(A_n)}{\vartheta_n}}.   \label{eqn:max_triangle}
\end{multline}
Since $C_r(A_n)$ is non-decreasing, we find from Lemma~\ref{lem:prob_difference} that
\[
\max_{n_k\le n\le n_{k+1}}\biggabs{\frac{C_r(A_{n_{k+1}})}{\vartheta_{n_{k+1}}}-\frac{C_r(A_n)}{\vartheta_{n_{k+1}}}}\le\sqrt{\frac{5(n_{k+1}-n_k)}{n_{k+1}}}
\]
almost surely for all sufficiently large $k$. From
\begin{equation}
\lim_{k\to\infty}\frac{n_{k+1}}{n_k}=\lim_{k\to\infty}e^{(k+1)^{1/2}-k^{1/2}}=1   \label{eqn:lim_nk_ratio}
\end{equation}
we conclude that, as $k\to\infty$,
\begin{equation}
\max_{n_k\le n\le n_{k+1}}\biggabs{\frac{C_r(A_{n_{k+1}})}{\vartheta_{n_{k+1}}}-\frac{C_r(A_n)}{\vartheta_{n_{k+1}}}}\to0 \quad\text{almost surely}.   \label{eqn:asc_2}
\end{equation}
The third term on the right hand side of~\eqref{eqn:max_triangle} can be bounded as
\[
\max_{n_k\le n\le n_{k+1}}\biggabs{\frac{C_r(A_n)}{\vartheta_{n_{k+1}}}-\frac{C_r(A_n)}{\vartheta_n}}\le \frac{C_r(A_{n_{k+1}})}{\vartheta_{n_{k+1}}}\;\biggabs{1-\frac{\vartheta_{n_{k+1}}}{\vartheta_{n_k}}}.
\]
Using~\eqref{eqn:lim_nk_ratio}, it is readily verified that
\[
\lim_{k\to\infty}\frac{\vartheta_{n_{k+1}}}{\vartheta_{n_k}}=1
\]
and, after combination with~\eqref{eqn:almost_sure_conv_nk}, we conclude that, as $k\to\infty$,
\begin{equation}
\max_{n_k\le n\le n_{k+1}}\
\biggabs{\frac{C_r(A_n)}{\vartheta_{n_{k+1}}}-\frac{C_r(A_n)}{\vartheta_n}}\to0\quad\text{almost surely}.   \label{eqn:asc_3}
\end{equation}
The required convergence~\eqref{eqn:almost_sure_conv_max} follows by combining~\eqref{eqn:max_triangle},~\eqref{eqn:almost_sure_conv_nk},~\eqref{eqn:asc_2}, and~\eqref{eqn:asc_3}.
\end{proof}


\section{Minimum values}
\label{sec:minimal}

Recall that the \emph{scalar product} between two vectors $x=(x_1,\dots,x_\ell)$ and $y=(y_1,\dots,y_\ell)$ in $\C^\ell$ is $\langle x,y\rangle=\sum_{j=1}^\ell x_j\overline{y_j}$, where bar means complex conjugation. We shall see that Theorems C and~\ref{thm:lower_bound_max} follow from well known results on the maximum magnitude of the nontrivial scalar products over a set of vectors in $\C^\ell$; a good overview is given by Kumar and Liu~\cite{Kumar1990}. The most famous such result is the following bound due to Welch~\cite{Welch1974}. 
\begin{lemma}[{Welch~\cite{Welch1974}}]
\label{lem:Welch_bound}
For positive integers $\ell$ and $m\ge 2$, let $v_1,\dots,v_m$ be elements of $\mathbb{C}^\ell$ satisfying $\|v_i\|_2^2=\ell$ for each $i$. Then, for integral $k\ge 1$,
\[
\max_{i\ne i'}\,\bigabs{\langle v_i,v_{i'}\rangle}\ge\Bigg[\frac{\ell^{2k}}{m-1}\Bigg(\frac{m}{{\ell+k-1\choose k}}-1\Bigg)\Bigg]^{1/{2k}}.
\]
\end{lemma}
\par
This lemma can be proved by observing
\[
m\ell^{2k}+m(m-1)\max_{i\ne i'}\,\bigabs{\langle v_i,v_{i'}\rangle}^{2k}\ge \sum_{i,i'}\bigabs{\langle v_i,v_{i'}\rangle}^{2k}
\]
and deriving a lower bound for the right hand side. We remark that, for $k>1$ and when the vectors have entries in $\{-1,1\}$, the bound in Lemma~\ref{lem:Welch_bound} can be slightly improved by a bound due to Sidelnikov~\cite{Sidelnikov1971}. Lemma~\ref{lem:Welch_bound} is now used to give a straightforward proof of Theorem~C.
\begin{proof}[Proof of Theorem~C]
Let $A_n=(a_1,a_2,\dots,a_n)$ be an element of $\{-1,1\}^n$. Write $\ell=\lfloor n/(2r+1)\rfloor$. For $\ell=0$, the theorem is trivial, so assume that $\ell\ge 1$. Let $S_1,S_2,\dots,S_m$ be $m=\lfloor (n-\ell+1)/r\rfloor$ pairwise disjoint $r$-element subsets of $\{0,\dots,n-\ell\}$. For each such set $S_i$, define the vector $v_i=(v_{i,1},\dots,v_{i,\ell})$ by
\[
v_{i,j}=\prod_{x\in S_i}a_{j+x}\quad\text{for each $j\in \{1,\dots,\ell\}$}.
\]
Since all of the sets $S_1,\dots,S_m$ have size $r$ and are pairwise disjoint, we have
\begin{equation}
C_{2r}(A_n)\ge \max_{i\ne i'}\,\bigabs{\langle v_i,v_{i'}\rangle}.   \label{eqn:inner_products_1}
\end{equation}
Observe that
\[
m=\bigg\lfloor \frac{n-\lfloor n/(2r+1)\rfloor+1}{r}\bigg\rfloor\ge \bigg\lfloor \frac{2n}{2r+1}\bigg\rfloor\ge 2\ell.
\]
Hence, $m\ge 2$ and we can apply Lemma~\ref{lem:Welch_bound} with $k=1$ to~\eqref{eqn:inner_products_1} to conclude
\[
\big[C_{2r}(A_n)\big]^2\ge \frac{\ell^2}{m-1}\Big(\frac{m}{\ell}-1\Big)>\ell\Big(1-\frac{\ell}{m}\Big)\ge \frac{\ell}{2},
\]
as required.
\end{proof}
\par
Slight improvements of Theorem~C are possible for particular values $r$, by choosing $\ell$ more carefully in the proof (see Anantharam~\cite{Anantharam2008} for $r=2$).
\par
We now prove Theorem~\ref{thm:lower_bound_max}.
\begin{proof}[Proof of Theorem~\ref{thm:lower_bound_max}]
Let $A_n=(a_1,a_2,\dots,a_n)$ be an element of $\{-1,1\}^n$. We have $n\ge 3$. Let $\ell=\lfloor n/3\rfloor$ and let $S_1,S_2,\dots,S_m$ be all $m={n-\ell+1\choose s}$ $s$-element subsets of $\{0,1,\dots,n-\ell\}$. For each such set $S_i$, define the vector $v_i=(v_{i,1},\dots,v_{i,\ell})$ by
\[
v_{i,j}=\prod_{x\in S_i}a_{j+x}\quad\text{for each $j\in \{1,\dots,\ell\}$}.
\]
Then
\[
\max\big\{C_2(A_n),C_4(A_n),\dots,C_{2s}(A_n)\big\}\ge \max_{i\ne i'}\,\bigabs{\langle v_i,v_{i'}\rangle}.
\]
We apply Lemma~\ref{lem:Welch_bound} with $k=s$ to get
\[
\Big[\max\big\{C_2(A_n),C_4(A_n),\dots,C_{2s}(A_n)\big\}\Big]^{2s}\ge \frac{\ell^{2s}}{m-1}\Bigg(\frac{m}{{\ell+s-1\choose s}}-1\Bigg).
\]
Write $n=3\ell+\delta$ for some $\delta\in\{0,1,2\}$. Then by~\eqref{eqn:binomial_bound} the leading term on the right hand side is
\[
\frac{\ell^{2s}}{m-1}\ge\frac{\ell^{2s}}{m}=\frac{(\frac{n-\delta}{3})^{2s}}{{(2n+\delta+3)/3\choose s}}\ge\bigg(\frac{s(n-\delta)^2}{3e(2n+\delta+3)}\bigg)^s>\bigg(\frac{sn}{9^2}\bigg)^s,
\]
using $n\ge 3$ and distinguishing the cases that $n\in\{3,4,5,6\}$ and $n\ge 7$ to get the last inequality.
\par
We complete the proof by showing that $m/{\ell+s-1\choose s}-1$ is greater than $1$. Define $f:\{1,2,\dots,\lfloor n/3\rfloor\}\to\Q$ by
\[
f(s)=\frac{{n-\ell+1\choose s}}{{\ell+s-1\choose s}}.
\]
A standard calculation shows that $f$ is monotonically increasing for $s\le (n-2\ell+2)/2$ and is monotonically decreasing for $s\ge (n-2\ell+2)/2$. Therefore, the minimum value of $f(s)$ is either $f(1)$ or $f(\lfloor n/3\rfloor)=f(\ell)$. Moreover, we readily verify that $f(1)>2$ and
\[
f(\ell)\ge\frac{{2\ell+1\choose \ell}}{{2\ell-1\choose \ell}}
=\frac{2(2\ell+1)}{\ell+1}\ge 3.
\]
Hence $f$ satisfies $f(s)>2$ on its entire domain, as required.
\end{proof}



\begin{thebibliography}{10}

\bibitem{Ahlswede2008}
R.~Ahlswede, J.~Cassaigne, and A.~S{\'a}rk{\"o}zy, \emph{On the correlation of
  binary sequences}, Discrete Appl. Math. \textbf{156} (2008), no.~9,
  1478--1487.

\bibitem{Aistleitner2013}
C.~Aistleitner, \emph{On the limit distribution of the well-distribution
  measure of random binary sequences}, J. Theor. Nombres Bordeaux \textbf{25}
  (2013), no.~2, 245--259.

\bibitem{Aistleitner2014}
\bysame, \emph{On the limit distribution of the normality measure of random
  binary sequences}, Bull. London Math. Soc. \textbf{46} (2014), no.~5,
  968--980.

\bibitem{Alon2006}
N.~Alon, Y.~Kohayakawa, C.~Mauduit, C.~G. Moreira, and V.~R\"odl,
  \emph{Measures of pseudorandomness: {Minimal} values}, Combin. Probab.
  Comput. \textbf{15} (2006), no.~1-2, 1--29.

\bibitem{Alon2007}
\bysame, \emph{Measures of pseudorandomness: {Typical} values}, Proc. London
  Math. Soc. \textbf{95} (2007), no.~3, 778--812.

\bibitem{Alon2008}
N.~Alon and J.~H. Spencer, \emph{The probabilistic method}, 3rd ed., Wiley,
  Hoboken, New Jersey, 2008.

\bibitem{Anantharam2008}
V.~Anantharam, \emph{A technique to study the correlation measures of binary
  sequences}, Discrete Math. \textbf{308} (2008), no.~24, 6203--6209.

\bibitem{Cassaigne2002}
J.~Cassaigne, C.~Mauduit, and A.~S{\'a}rk{\"o}zy, \emph{On finite pseudorandom
  binary sequences. {VII}. {T}he measures pseudorandomness}, Acta Arith.
  \textbf{103} (2002), no.~2, 97--118.

\bibitem{Feller1968}
W.~Feller, \emph{An introduction to probability theory and its applications.
  {V}ol. {I}}, Third edition, John Wiley \& Sons Inc., New York, 1968.

\bibitem{Kumar1990}
P.~V. Kumar and C.~M. Liu, \emph{On lower bounds to the maximum correlation of
  complex roots-of-unity sequences}, IEEE Trans. Inf. Theory \textbf{36}
  (1990), no.~3, 633--640.

\bibitem{Mauduit1997}
C.~Mauduit and A.~S\'ark\"ozy, \emph{On finite pseudorandom binary sequences
  {I}: Measure of pseudorandomness, the {Legendre} symbol}, Acta Arith.
  \textbf{82} (1997), no.~4, 265--377.

\bibitem{McDiarmid1989}
C.~McDiarmid, \emph{On the method of bounded differences}, Surveys in
  Combinatorics (J.~Siemons, ed.), London Math. Soc. Lectures Notes Ser. 141,
  Cambridge Univ. Press, Cambridge, 1989, pp.~148--188.

\bibitem{Mercer2006}
I.~D. Mercer, \emph{Autocorrelations of random binary sequences}, Combin.
  Probab. Comput. \textbf{15} (2006), no.~5, 663--671.

\bibitem{Schmidt2014}
K.-U. Schmidt, \emph{The peak sidelobe level of random binary sequences},
  Bull.\ London Math.\ Soc. \textbf{46} (2014), no.~3, 643--652.

\bibitem{Sidelnikov1971}
V.~M. Sidel{\cprime}nikov, \emph{On mutual correlation of sequences}, Soviet
  Math. Dokl. \textbf{12} (1971), 197--201.

\bibitem{Welch1974}
L.~R. Welch, \emph{Lower bounds on the maximum cross correlation of signals},
  IEEE Trans. Inf. Theory \textbf{IT-20} (1974), no.~3, 397--399.

\end{thebibliography}

\def\cprime{$'$}
\providecommand{\bysame}{\leavevmode\hbox to3em{\hrulefill}\thinspace}
\providecommand{\MR}{\relax\ifhmode\unskip\space\fi MR }
\providecommand{\MRhref}[2]{%
  \href{http://www.ams.org/mathscinet-getitem?mr=#1}{#2}
}
\providecommand{\href}[2]{#2}

\end{document}